\documentclass[12 pt]{amsart}

\usepackage{etex}
\usepackage[shortlabels]{enumitem}
\usepackage{amsmath}
\usepackage{amscd}
\usepackage{amsthm}
\usepackage{amsfonts}
\usepackage{amssymb}
\usepackage{eucal}
\usepackage{mathrsfs}
\usepackage{morefloats}
\usepackage{thm-restate}
\usepackage{hyperref}
\usepackage[margin=1in]{geometry}
\usepackage{tikz}
\usepackage{tikz-cd, comment, fancyvrb, alltt, changepage}

\theoremstyle{plain}
\newtheorem{theorem}{Theorem}
\newtheorem{proposition}[theorem]{Proposition}
\newtheorem{lemma}[theorem]{Lemma}

\theoremstyle{definition}
\newtheorem{definition}{Definition}

\newtheorem*{theorem*}{Theorem}
\newtheorem*{proposition*}{Proposition}
\newtheorem*{lemma*}{Lemma}

\theoremstyle{remark}
\newtheorem{remark}[definition]{Remark}

\numberwithin{equation}{section}

\newcommand{\ZZ}{\mathbb Z}

\newcommand{\on}{\operatorname}
\newcommand{\ol}{\overline}

\newcommand{\wh}{\widehat}
\makeatletter
\newcommand*{\defeq}{\mathrel{\rlap{%
                     \raisebox{0.3ex}{$\m@th\cdot$}}%
                     \raisebox{-0.3ex}{$\m@th\cdot$}}%
                     =}
\makeatother

\newcommand{\id}{\on{id}}
\newcommand\ab{\on{ab}}

\DeclareMathOperator\Sp{Sp}

\newcommand\nc{\newcommand}
\nc\renc{\renewcommand}

\newcommand\bz{{\mathbb Z}}

\makeatletter
\newcommand{\customlabel}[2]{%
   \protected@write \@auxout {}{\string \newlabel {#1}{{#2}{\thepage}{#2}{#1}{}} }%
   \hypertarget{#1}{#2}
}

\DeclareMathOperator\matone{\Phi}
\DeclareMathOperator\mattwo{\Psi}

\newcommand{\mf}{\mathfrak}

\usepackage[colorinlistoftodos, textsize=tiny]{todonotes}
\makeatletter
\providecommand\@dotsep{5}
\def\listtodoname{List of Todos}
\def\listoftodos{\@starttoc{tdo}\listtodoname}
\makeatother

\newcommand{\zh}{\wh{\bz}}
\DeclareMathOperator\SL{SL}

\title[Lifting Subgroups of Symplectic Groups over $\mathbb{Z}/\ell\mathbb{Z}$]{Lifting Subgroups of Symplectic Groups over $\mathbb{Z}/\ell\mathbb{Z}$}

\date{\today}

\author[Aaron Landesman]{Aaron Landesman}

\author[Ashvin A. Swaminathan]{Ashvin A. Swaminathan}

\author[James Tao]{James Tao}

\author[Yujie Xu]{Yujie Xu}

\AtEndDocument{\bigskip{\footnotesize%
  \textsc{Department of Mathematics, Stanford University, \mbox{Stanford, CA 94305}} \par  
  \textit{E-mail address}, Aaron Landesman: \texttt{aaronlandesman@stanford.edu} \par
  \addvspace{\medskipamount}
  \textsc{Department of Mathematics, Harvard College, \mbox{Cambridge, MA 02138}} \par  
  \textit{E-mail address}, Ashvin A. Swaminathan: \texttt{aaswaminathan@college.harvard.edu} \par
  \addvspace{\medskipamount}
  \textsc{Department of Mathematics, Harvard College, \mbox{Cambridge, MA 02138}} \par  
\textit{E-mail address}, James Tao: \texttt{jamestao@college.harvard.edu}
\par
  \addvspace{\medskipamount}
  \textsc{Department of Mathematics, Caltech, \mbox{Pasadena, CA 91126}} \par
  \textit{E-mail address}, Yujie Xu: \texttt{yujiex@caltech.edu}
}}
\begin{document}

\begin{abstract}
For a positive integer $g$, let $\mathrm{Sp}_{2g}(R)$ denote the group of $2g \times 2g$ symplectic matrices over a ring $R$. Assume $g \ge 2$. For a prime number $\ell$, we show that any closed subgroup of $\mathrm{Sp}_{2g}(\mathbb{Z}_\ell)$ that surjects onto $\mathrm{Sp}_{2g}(\mathbb{Z}/\ell\mathbb{Z})$ must in fact equal all of $\mathrm{Sp}_{2g}(\mathbb{Z}_\ell)$. Our result is motivated by group theoretic considerations that arise in the study of Galois representations associated to abelian varieties.
\end{abstract}
\maketitle

\vspace*{-0.1in}
\section{Introduction}

Let $g$ be a positive integer, and for a ring $R$, denote by $\mathrm{Sp}_{2g}(R)$ the group of $2g \times 2g$ symplectic matrices over $R$. Let $\ZZ_\ell$ denote the ring of $\ell$-adic integers, and consider the natural projection map $\Sp_{2g}(\ZZ_\ell) \to \Sp_{2g}(\ZZ/\ell \ZZ)$. 
In this paper, we show that when $g > 1$, there are no proper closed subgroups of $\Sp_{2g}(\ZZ_\ell)$ that surject via this projection map onto all of $\Sp_{2g}(\ZZ/\ell \ZZ)$.

The case $g = 1$, in which $\Sp_2 = \SL_2$, is well-understood. Indeed,
as proven in~\cite[Lemma 3, Section IV.3.4]{serre1989abelian},
if $\ell \geq 5$ and $H_\ell \subset \SL_2(\ZZ_\ell)$ is a closed subgroup that surjects onto $\SL_2(\ZZ/\ell \ZZ)$, then $H_\ell = \SL_2(\ZZ_\ell)$. The corresponding result for $\ell \in \{2,3\}$ simply does not hold: in each of these cases, there are nontrivial subgroups of $\Sp_{2g}(\ZZ / \ell^2 \ZZ)$ that surject onto $\SL_2(\ZZ/\ell \ZZ)$. See~\cite[Section IV.3.4, Exercises 1 -- 3]{serre1989abelian} for exercises outlining a proof, and also see for more comprehensive descriptions~\cite{dontyouknowit} for the case $\ell = 2$ and~\cite{onething} for the case $\ell = 3$.

The objective of the present article is to generalize~\cite[Lemma 3, Section IV.3.4]{serre1989abelian},
 to hold for all $g \geq 2$. Our main theorem is \mbox{stated as follows:}
\begin{theorem}\label{jamesdreamcometrue}
Let $g \ge 2$, let $\ell$ be a prime number, and let $H_\ell \subset \Sp_{2g}(\bz_\ell)$ be a closed subgroup. If the mod-$\ell$ reduction of $H_\ell$ equals all of $\Sp_{2g}(\bz / \ell \ZZ)$, then $H_\ell = \Sp_{2g}(\bz_\ell)$, and \mbox{in particular, the} mod-$\ell^k$ reduction of $H_\ell$ equals all of $\Sp_{2g}(\ZZ/\ell^k \ZZ)$ for each \mbox{positive integer $k$.} 
\end{theorem}

\begin{remark}
A more general version of Theorem~\ref{jamesdreamcometrue} is proven for a large class of semisimple Lie groups $G$ in
\cite[Theorem B]{weigel:on-the-profinite-completion-of-arithmetic-groups-of-split-type} (except for the case that $g = 3$ and $\ell = 2$)
and also in
~\cite[Theorem 1.3]{vasiu2003surjectivity}. In the present article, we provide an elementary and self-contained proof for the special case $G = \Sp_{2g}$. In particular, our inductive method circumvents the use of Lie theory, and is therefore suitable for somewhat more general groups (e.g.\ certain finite-index subgroups of matrix groups over $\bz_\ell$) which arise in the study of Galois representations associated to abelian varieties, cf.\ \cite{landesman-swaminathan-tao-xu:hyperelliptic-curves}. 
\end{remark}

\begin{remark}
To give a typical application, one can directly use Theorem~\ref{jamesdreamcometrue} to reduce the problem of checking that the $\ell$-adic Galois representation associated to an abelian variety has maximal image to the simpler problem of checking that the mod-$\ell$ reduction has maximal image. Indeed, the conclusion of Theorem~\ref{jamesdreamcometrue} has been applied many times in the study of Galois representations, such as
	in~\cite[Proof of Lemma 2.4]{zywina2015example},
	~\cite[Proof of Theorem 8]{Hui2016},
	~\cite[Proof of Corollary 3.5]{achter-pries:integral-monodromy-of-hyperelliptic-and-trielliptic-curves},~\cite[Proof of Lemma 5.1]{achter-pries:monodromy-of-p-rank-strat-of-the-moduli-space-of-curves},~\cite[p.~467]{scoopdedoo},
	as well in the authors' own papers~\cite{landesman-swaminathan-tao-xu:rational-families} and~\cite{landesman-swaminathan-tao-xu:hyperelliptic-curves}. 
\end{remark}

The rest of this paper is organized as follows. In Section~\ref{subsection:stimpy}, we introduce the basic definitions and properties of the symplectic group. Next, in Section~\ref{subsection:abelianizations}, we compute the commutator subgroups of $\Sp_{2g}(\ZZ_\ell)$ and $\Sp_{2g}(\ZZ/\ell^k \ZZ)$ for every prime number $\ell$ and positive integer $k$. Finally, in Section~\ref{doyoulift}, we prove Theorem~\ref{jamesdreamcometrue}.
 
\section{Background on Symplectic Groups}

In this section, we first detail the basic definitions and properties of symplectic groups, and then proceed to prove a number of lemmas that are used in our proof of Theorem~\ref{jamesdreamcometrue}.

\subsection{Symplectic Groups}\label{subsection:stimpy}
Fix a commutative ring $R$,
let $\on{Mat}_{2g \times 2g}(R)$ denote the space of $2g \times 2g$ matrices with entries in $R$, and let $\Omega_{2g} \in \on{Mat}_{2g \times 2g}(R)$ be defined by
$$\Omega_{2g} \defeq \left[\begin{array}{c|c} 0 & \id_g \\ \hline -\id_g & 0\end{array}\right],$$
where $\id_g$ denotes the $g \times g$ identity matrix. 
We define 
the symplectic group $\Sp_{2g}(R)$ as the set of $M \in \SL_{2g}(R)$ so that $M^T \Omega_{2g} M = \Omega_{2g}$.

In the proof of Theorem~\ref{jamesdreamcometrue}, we will make heavy use of the ``Lie algebra'' $\mf{sp}_{2g}(R)$, which is defined by 
		\begin{align*} 
			\mf{sp}_{2g}(R) &\defeq \{M \in \on{Mat}_{2g \times 2g}(R) : M^T \Omega_{2g} + \Omega_{2g} M = 0 \}. 
		\end{align*} 
		It is easy to see that $M^T \Omega_{2g} + \Omega_{2g}M = 0$ is equivalent to $M$ being a block matrix with $g \times g$ blocks of the form
		\[
			M = \left[\begin{array}{c|c} A & B \\ \hline C & -A^T \end{array}\right], 
		\]
		where $B$ and $C$ are symmetric.

In what follows, we specialize to studying symplectic groups over $R = \ZZ$, $R = \ZZ_{\ell}$, or $R = \ZZ / \ell^k \ZZ$ for $\ell$ a prime number and $k$ a positive integer. We will adhere to the following notational conventions:
\begin{itemize}
\item Let $H_{\ell} \subset \Sp_{2g}(\ZZ_\ell)$ be a closed subgroup.
\item Let $H(\ell^k) \subset \Sp_{2g}(\ZZ/\ell^k \ZZ)$ be the mod-$\ell^k$ reduction of $H_\ell$.
\item Notice that the map $S \mapsto \id_{2g} + \ell^k S$ gives an isomorphism of groups
$$\mf{sp}_{2g}(\ZZ/ \ell \ZZ) \simeq \ker(\Sp_{2g}(\ZZ/\ell^{k+1} \ZZ) \to \Sp_{2g}(\ZZ/ \ell^k \ZZ))$$
for every $k \geq 1$. We will use the Lie algebra notation $\mf{sp}_{2g}(\mathbb Z/\ell\ZZ)$ to denote the above kernel when we want to think of its elements additively, and we will use the kernel notation when we want to view its elements multiplicatively.
\item For any group $G$, let $[G,G]$ be its commutator subgroup, and let $G^{\ab} = G/[G,G]$ be its abelianization.
\end{itemize}

\subsection{Commutators}
\label{subsection:abelianizations}

We shall now compute the abelianizations of $\Sp_{2g}(\bz_\ell)$ and $\Sp_{2g}(\bz / \ell^k \ZZ)$ for every integer $g \geq 2$, prime number $\ell$, and positive integer $k$. It will first be convenient for us to compute the abelianization $\Sp_{2g}(\ZZ)^{\ab}$.

\begin{lemma} \label{theorem:commutator-sp-z}
The group $\Sp_{2g}(\bz)^{\ab}$ is trivial when $g \geq 3$ and is isomorphic to \mbox{$\mathbb Z/2 \ZZ$ when $g = 2$.}
\end{lemma}
\begin{proof}
	The case $g \geq 3$ follows from~\cite[Remark, p.\ 123]{bass1967solution}, so it only remains to deal with the case $g = 2$. By~\cite[Satz]{bender1980presentation}, $\Sp_{4}(\mathbb Z)$ has two generators $K$ and $L$ that satisfy several relations, three of which are given as follows:
\begin{align*}
K^2 &= \id_{2g}, \\
L^{12} &= \id_{2g}, \\
(K\cdot L^5)^5 &= (L^6\cdot K\cdot L^5\cdot K\cdot L^7\cdot K)^2.
\end{align*}
By the universal property of the abelianization, we have that $\Sp_4(\ZZ)^{\ab}$ is a quotient of the rank-$2$ free abelian group $K \mathbb Z \oplus L \mathbb Z$ generated by $K$ and $L$. Thus, from the aforementioned multiplicative relations between $K$ and $L$ in $\Sp_4(\mathbb Z)$, we obtain the following additive relations in the abelianization
\begin{align*}
2K &= 0, \\
12L &= 0, \\
5(K + 5L) &= 2\cdot (6L + K + 5L + K + 7L + K).
\end{align*}
Substituting the first two relations above into the third relation, we find $L = K$, which implies that $\Sp_4(\ZZ)^{\ab}$ is a quotient of $(K \mathbb Z \oplus L \mathbb Z)/(2K, K - L) \simeq \mathbb Z / 2\bz$.
		
It remains to show that $\Sp_{4}(\mathbb Z)$ maps surjectively onto $\mathbb Z/2\bz$. Postcomposing the surjection $\Sp_{4}(\mathbb Z) \twoheadrightarrow \Sp_{4}(\mathbb Z/2 \ZZ)$ with the isomorphism $\Sp_{4}(\mathbb Z/2 \ZZ) \simeq S_6$ by \cite[3.1.5]{omeara1978symplectic} and then applying the sign map $S_6 \twoheadrightarrow \mathbb Z/2 \ZZ$ \mbox{yields the desired result.} \end{proof}

\begin{remark}
Let $\wh{\ZZ}$ denote the profinite completion of $\ZZ$. It follows immediately from Lemma~\ref{theorem:commutator-sp-z}, together with the fact if $G_1$ and $G_2$ are groups with $\wh{G_1} \simeq \wh{G_2}$ then $\wh{G_1^{\ab}} \simeq \wh{G_2^{\ab}}$, that the group $\Sp_{2g}(\widehat {\mathbb Z})^{\ab}$ is trivial for $g \geq 3$ and is isomorphic to \mbox{$\ZZ / 2 \ZZ$ when $g = 2$.}
\end{remark}

Using Lemma~\ref{theorem:commutator-sp-z}, we can now compute
the abelianizations of all aforementioned groups.

\begin{proposition} \label{corollary:finite-commutator}
We have the following results:
\begin{enumerate}
\item The group $\Sp_{2g}(\mathbb Z_\ell)^{\ab}$ is trivial except when $g = \ell = 2$, in which case it is isomorphic to $\ZZ / 2 \ZZ$.
\item Let $k \ge 1$. The group $\Sp_{2g}(\mathbb Z/\ell^k \ZZ)^{\ab}$ is trivial except when $g = \ell = 2$, in which case it is isomorphic to $\ZZ / 2 \ZZ$.
\end{enumerate}
\end{proposition}
\begin{proof}
We first verify Statement (a). For $g \ge 3$, since $\Sp_{2g}(\wh{\mathbb Z})^{\ab}$ is trivial, and since we have a surjection $\Sp_{2g}(\zh) \twoheadrightarrow \Sp_{2g}(\mathbb Z_\ell)$, the result follows immediately. 
Now take $g = 2$. 
First, we have surjections
$\Sp_4(\widehat{\mathbb Z}) \twoheadrightarrow \Sp_4(\mathbb Z_2) \twoheadrightarrow \Sp_4(\mathbb Z/2 \ZZ) \cong S_6$.
Since the former and the latter have abelianizations isomorphic
to $\mathbb Z/2 \ZZ$, using Lemma~\ref{theorem:commutator-sp-z}, it follows
that $\Sp_4(\mathbb Z_2)^{\ab} \cong \mathbb Z/2 \ZZ$.
Then, since we have
\begin{align*}
	\mathbb Z/2 \ZZ \cong \Sp_4(\widehat{\mathbb Z})^{\ab} \cong \Sp_4(\mathbb Z_2)^{\ab} \times \prod_{\ell\neq 2} \Sp_4(\mathbb Z_\ell)^{\ab} \cong \mathbb Z/2\ZZ \times \prod_{\ell\neq 2} \Sp_4(\mathbb Z_\ell)^{\ab},
\end{align*}
it follows that $\Sp_4(\mathbb Z_\ell)^{\ab}$ is trivial for $\ell \neq 2$.

Now, observe that Statement (b) follows from Statement (a): indeed, notice that the surjection $\Sp_{2g}(\mathbb Z_\ell) \rightarrow \Sp_{2g}(\mathbb Z/ \ell^k \ZZ)$ induces a surjection $\Sp_{2g}(\mathbb Z_\ell)^{\ab} \rightarrow \Sp_{2g}(\mathbb Z/ \ell^k \ZZ)^{\ab}$ and in the case
of $\ell = g = 2$, $\Sp_4(\mathbb Z/ 2^k \ZZ)^{\ab}$ is nontrivial as it surjects
onto $\Sp_4(\mathbb Z/ 2 \ZZ)^{\ab} \cong \mathbb Z/2 \ZZ$.
\end{proof}

\section{Proof of Theorem~\ref{jamesdreamcometrue}}\label{doyoulift}
In this section, we provide a complete proof of the main theorem of this paper, namely Theorem~\ref{jamesdreamcometrue}. The basic strategy has two steps: lift from $\ell^2$ to $\ell^\infty$ (see Section~\ref{allthewaytonight}), and lift from $\ell$ to $\ell^2$ (see Section~\ref{bataneye}). Considerable care must be taken in dealing with the cases where $\ell = 2, 3$, so we treat these situations separately (see Sections~\ref{booyah} and~\ref{thisishisbananas}). We execute this strategy as follows:
	
	\subsection{Lifting from $\ell^2$ to $\ell^\infty$ for $\ell \ge 3$, and from $8$ to $2^\infty$} \label{allthewaytonight}
	
	\begin{lemma} \label{proposition:lifting-l-squared}
		If $\ell \ge 3$, then $H(\ell^2) = \Sp_{2g}(\bz / \ell^2 \ZZ)$ implies $H_\ell = \Sp_{2g}(\bz_\ell)$. If $\ell = 2$, then $H(8) = \Sp_{2g}(\bz / 8 \ZZ)$ implies $H_2 = \Sp_{2g}(\bz_2)$.  
	\end{lemma}
	\begin{proof}
		This is done in the $g = 1$ case in in~\cite[Lemma 3, Section IV.3.4]{serre1989abelian}, which readily generalizes to the case of $g \geq 2$.
	\end{proof}
	
	\subsection{Lifting from $\ell$ to $\ell^2$ for $\ell \ge 5$}\label{bataneye}
	
	\begin{lemma} \label{proposition:lifting-l}
		Fix $g \ge 2$. If $\ell \ge 5$, then $H(\ell) = \Sp_{2g}(\bz / \ell \ZZ)$ implies $H(\ell^2) \subset \Sp_{2g}(\bz / \ell^2 \ZZ)$. 
	\end{lemma}
	\begin{proof}
		It suffices to show that $H(\ell^2)$ contains all of 
		\[
			\ker( \Sp_{2g}(\bz / \ell^2 \ZZ) \twoheadrightarrow \Sp_{2g}(\bz / \ell \ZZ)) = \id_{2g} + \ell \cdot \mf{sp}_{2g}(\bz / \ell \ZZ). 
		\]
		We prove this by taking $\ell$-th powers of specific matrices. We want to pick $M \in \mf{sp}_{2g}(\bz / \ell \ZZ)$ such that $\id_{2g} + M$ lies in $\Sp_{2g}(\bz / \ell \ZZ)$ and such that $M^2 = 0$. As it happens, these two conditions are equivalent: indeed, since $M^T \Omega_{2g} + \Omega_{2g}M = 0$, we have 
		\begin{align*} 
			(\id_{2g} + M)^T \Omega_{2g} (\id_{2g} + M) &= \Omega_{2g} + M^T \Omega_{2g} + \Omega_{2g}M + M^T \Omega_{2g} M \\
			&= \Omega_{2g} + M^T \Omega_{2g} M \\
			&= \Omega_{2g} - \Omega_{2g}M^2, 
		\end{align*}
		so the condition that $\id_{2g} + M$ is symplectic is equivalent to the condition that $M^2 = 0$. Moreover, by expanding in terms of matrices, we see that $M^2 = 0$ if and only if
		\[
			\left[\begin{array}{c|c} A^2 + BC & AB - BA^T \\ \hline CA - A^T C & CB + (A^T)^2 \end{array}\right] = 0.
		\]
		So, if $A^2 = 0$ and two of $A, B,$ and $C$ are zero, the matrix $\id_{2g} + M$ will be symplectic.
		
		For $M \in \mf{sp}_{2g}(\ZZ/\ell \ZZ)$, choose an arbitrary lift of $M$ in $\on{Mat}_{2g \times 2g}(\bz / \ell^2 \ZZ)$, and by abuse of notation, also denote it $M$. By assumption, $H(\ell^2)$ contains an element of the form $\id_{2g} + M + \ell V$ for some $V \in \on{Mat}_{2g \times 2g}(\bz / \ell \ZZ)$. This means that $H(\ell^2)$ contains
		\begin{align*} 
			(\id_{2g} + M + \ell V)^\ell &\equiv \id_{2g} + \ell(M + \ell V) + \binom{\ell}{2} (M + \ell V)^2 + \cdots + \binom{\ell}{\ell - 1} (M + \ell V)^{\ell - 1} + (M + \ell V)^\ell \\
			&\equiv \id_{2g} + \ell M \pmod{\ell^2}, 
		\end{align*} 
        where the last step above relies crucially upon the assumption that $\ell \geq 5$.
We conclude that $H(\ell^2)$ contains $\id_{2g} + \ell M$ for every $M$ satisfying the above conditions. 
		
		Taking $A = C = 0$, we see that $H(\ell^2)$ contains 
		\[
			\id_{2g} + \ell \cdot \left[\begin{array}{c|c} 0 & B \\ \hline 0 & 0 \end{array}\right] 
		\]
		for any symmetric matrix $B$. Similarly, taking $A = B = 0$ shows that $H(\ell^2)$ contains 
		\[
			\id_{2g} + \ell \cdot \left[\begin{array}{c|c} 0 & 0 \\ \hline C & 0 \end{array}\right]
		\]
		for any symmetric matrix $C$. Taking $B = C = 0$ shows that $H(\ell^2)$ contains
		\[
			\id_{2g} + \ell \cdot \left[\begin{array}{c|c} A & 0 \\ \hline 0 & -A^T \end{array}\right]
		\]
		for any matrix $A$ with $A^2 = 0$. It is a standard fact that the span of such matrices $A$ is the space of trace zero matrices. Observe that $\on{tr} A = 0$ is a single linear condition on $\mf{sp}_{2g}(\bz / \ell \ZZ)$ that singles out a codimension-one linear subspace $W \subset \mf{sp}_{2g}(\bz / \ell \ZZ)$, so that 
		\begin{equation}\label{searstowerprom}
			\id_{2g} + \ell W \subset H(\ell^2) \cap \ker ( \Sp_{2g}(\bz / \ell^2 \ZZ) \twoheadrightarrow \Sp_{2g}(\bz / \ell \ZZ)). 
		\end{equation}
		Observe that if the inclusion in~\eqref{searstowerprom} were strict, then the right-hand side would be all of $\id_{2g} + \ell \cdot \mf{sp}_{2g}(\bz / \ell \ZZ)$, and the desired result follows. Therefore, suppose the inclusion in~\eqref{searstowerprom} is an equality. Since $W$ has index $\ell$ in $\mf{sp}_{2g}(\bz / \ell \ZZ)$, it follows that $H(\ell^2)$ has index $\ell$ in $\Sp_{2g}(\bz / \ell^2 \ZZ)$. We obtain a surjection 
		\[
			\Sp_{2g}(\bz / \ell^2 \ZZ) \twoheadrightarrow \Sp_{2g}(\bz / \ell^2 \ZZ) / H(\ell^2) \cong \bz / \ell \ZZ, 
		\]
		which contradicts Proposition~\ref{corollary:finite-commutator}. 
	\end{proof} 
	
	\subsection{Lifting from $4$ to $8$} \label{booyah}
	
	\begin{lemma} \label{proposition:lifting-4}
		Fix $g \ge 2$. Then $H(4) = \Sp_{2g}(\bz / 4 \ZZ)$ implies $H(8) = \Sp_{2g}(\bz / 8 \ZZ)$. 
	\end{lemma}
	\begin{proof} 
		We modify the proof of Lemma~\ref{proposition:lifting-l}. As in that proof, we consider a matrix
		\[
			M  = \left[\begin{array}{c|c} A & B \\ \hline C & -A^T \end{array}\right] \in \mf{sp}_{2g}(\bz / 2 \ZZ) 
		\]
		with the property that $\id_{2g} + 2M$ lies in $\Sp_{2g}(\bz / 4 \ZZ)$ and $M^2 = 0$. This time, however, the first condition automatically holds because 
		\begin{align*} 
			(\id_{2g} + 2M)^T \Omega_{2g}(\id_{2g} + 2M) &\equiv \Omega_{2g} + 2(M^T \Omega_{2g} + \Omega_{2g}M) + 4M^T \Omega_{2g}M \\
			&\equiv \Omega_{2g} \pmod{4}.
		\end{align*} 
		Nevertheless, note that the second condition is again satisfied whenever $A^2 = 0$ and two of $A, B,$ and $C$ are zero. 
		
		Choose an arbitrary lift of $M$ in $\on{Mat}_{2g \times 2g}(\bz / 4 \ZZ)$, and by abuse of notation also refer to it as $M$. By assumption, $H(8)$ contains an element of the form $\id_{2g} + 2M + 4V$ for some $V \in \on{Mat}_{2g \times 2g}(\bz / 2 \ZZ)$, which means that $H(8)$ contains 
		\begin{align*} 
			(\id_{2g} + 2M + 4V)^2 &\equiv 
			\id_{2g} \mathbin{+} 4M \pmod{8}. 
		\end{align*} 
		Taking $W \subset \mf{sp}_{2g}(\bz / 2 \ZZ)$ to be the trace-zero subspace as before, this implies that 
		\[
			\id_{2g} \mathbin{+} 4 \cdot W \subset H(8) \cap \ker (\Sp_{2g}(\bz / 8 \ZZ) \twoheadrightarrow \Sp_{2g}(\bz / 4 \ZZ) ). 
		\]
		If the inclusion were strict, again the right-hand side would contain the kernel of reduction, so the desired result follows. Therefore, suppose the inclusion is an equality, so that $H(8)$ has index $2$ in $\Sp_{2g}(\bz / 8 \ZZ)$. If $g \ge 3$, Proposition~\ref{corollary:finite-commutator} tells us that $\Sp_{2g}(\bz / 8 \ZZ)^{\ab}$ is trivial, which is a contradiction. If $g = 2$, the same proposition tells us that $H(8) = [\Sp_4(\ZZ/ 8 \ZZ),\Sp_4(\ZZ/8\ZZ)]$, so since the image of this commutator under the abelianization map
		\[
			\Sp_{4}(\bz / 8 \ZZ) \twoheadrightarrow \Sp_{4}(\bz / 2 \ZZ) \twoheadrightarrow \ZZ/2\ZZ
		\]
        is trivial, $H(8)$ cannot surject onto $\Sp_{4}(\bz / 2\ZZ)$, which is again a contradiction. 
	\end{proof} 
	
	\subsection{Lifting from 2 to 4 and from 3 to 9}\label{thisishisbananas}

	\begin{proposition} \label{proposition:lifting-2}
		Fix $g \ge 2$. The following statements hold: 
		\begin{enumerate} 
			\item Take $\ell = 2$. Then $H(2) = \Sp_{2g}(\bz / 2 \ZZ)$ implies that $H(4) = \Sp_{2g}(\bz / 4\ZZ)$. 
			\item Take $\ell = 3$. Then $H(3) = \Sp_{2g}(\bz / 3 \ZZ)$ implies that $H(9) = \Sp_{2g}(\bz / 9 \ZZ)$. 
		\end{enumerate} 
	\end{proposition} 

\subsubsection*{Idea of Proof}
The argument will proceed by induction on $g$. We start by verifying
the base case $g = 2$ in Lemma~\ref{lemma:base-case-low-lifting}.
We then inductively assume this holds for $g-1$ and prove it for $g$.
We use the inductive hypothesis to construct a particular element lying
in $H(\ell^2)$ in Lemma~\ref{lemma:particular-matrices-in-preimage}.
Then, we use Lemma~\ref{lemma:code} to show that conjugates of this
particular element generate $\mf{sp}_{4}(\mathbb Z/ \ell^2 \ZZ)$, embedded
as in ~\eqref{abovesubspace}. We finally translate around this
copy of $\mf{sp}_{4}(\mathbb Z/ \ell^2 \ZZ)$ to obtain that $H(\ell^2) = \Sp_{2g}(\mathbb Z/\ell^2 \ZZ)$.

The following lemma deals with the base case:
\begin{lemma}
	\label{lemma:base-case-low-lifting}
	Proposition~\ref{proposition:lifting-2} holds
	in the case that 
	$g = 2$.
\end{lemma}
\begin{proof}
		Let $\ell \in \{2,3\}$.	We proceed by induction on $g$. The following \texttt{Magma} code verifies the base cases for $g = 2$,
		that the only subgroup
of $\Sp_4(\mathbb Z/ \ell^2 \ZZ)$ surjecting onto $\Sp_4(\mathbb Z/ \ell \ZZ)$ is all of
$\Sp_4(\mathbb Z/\ell^2 \ZZ)$.
\vspace*{0.1in}
\begin{adjustwidth}{0.5in}{0in}
\begin{alltt}
for l in [2,3] do
    Z := Integers();
    G := GL(4,quo<Z|l*l>);
    A := elt<G | 1,0,0,0, 1,-1,0,0, 0,0,1,1, 0,0,0,-1>;
    B := elt<G| 0,0,-1,0, 0,0,0,-1, 1,0,1,0, 0,1,0,0>;
    H := sub<G|A,B>;
    maximals := SubgroupClasses(H: Al := "Maximal");
    S := quo<Z|l>;
    grp, f := ChangeRing(G, S);
    for H in maximals do
        if #f(H{\`{}}subgroup) eq #Sp(4,l) then
            assert false;
        end if;
    end for;
end for; 
\end{alltt}
\end{adjustwidth}
\vspace*{0.1in}
This concludes the proof of Lemma~\ref{lemma:base-case-low-lifting}.
\end{proof}

In order to handle the inductive step of Proposition~\ref{proposition:lifting-2}, we first introduce some notation.

Let 
		\(
		\phi_{\ell,2g} \colon \Sp_{2g}(\bz / \ell^2 \ZZ) \twoheadrightarrow \Sp_{2g}(\bz / \ell \ZZ)
		\)
		denote the usual reduction map. 
Our next aim is to define the maps $\pi, \iota_\ell$
in the diagram 
		\begin{equation}\label{roarrrr}
		\begin{tikzcd}[row sep=tiny]
		& \phi_{\ell,2g}^{-1}(\Sp_{2g-2}(\bz/ \ell \ZZ)) \ar[hookrightarrow]{r} \ar[twoheadrightarrow]{dd} \ar[twoheadrightarrow, swap]{ld}{\pi} & \Sp_{2g}(\bz / \ell^2 \ZZ) \ar[twoheadrightarrow]{dd}{\phi_{\ell,2g}} \\
      	\Sp_{2g-2}(\bz / \ell^2 \ZZ) \ar[twoheadrightarrow, swap]{rd}{\phi_{\ell,2g-2}} \\
		& \Sp_{2g-2}(\bz/\ell \ZZ) \ar[hookrightarrow, swap]{r}{\iota_\ell} & \Sp_{2g}(\bz/\ell \ZZ)
		\end{tikzcd}
		\end{equation}
The map $\iota_\ell$ will be defined in~\eqref{iota-definition}
and the map $\pi$ will be defined in~\eqref{pi-definition}
For the present purpose, it is convenient to use a different $\Omega$-matrix, which we shall denote by $J_{2g}$, in the definition of symplectic group $\Sp_{2g}$ and its Lie algebra $\mf{sp}_{2g}$. Inductively define
		\[
			J_2 \defeq \left[\begin{array}{cc} 0 & 1 \\ -1 & 0 \end{array}\right] \quad \text{and} \quad 	J_{2g} \defeq \left[\begin{array}{c|c} 
				J_{2g-2} & 0 \\ \hline 
				0 & J_2
				\end{array}\right]. 
		\]
		The matrix $J_{2g}$ is a block-diagonal matrix with each block being a copy of $J_2$. 
        
        Now let $M \in \Sp_{2g-2}(\bz / \ell^2 \ZZ)$. The map 
		\[
			M \mapsto \left[\begin{array}{c|c} M & 0 \\ \hline 0 & \id_2 \end{array}\right] 
		\]
		gives an inclusion 
		\[
			\iota_{\ell^2} \colon \Sp_{2g-2}(\bz /\ell^2 \ZZ) \hookrightarrow \Sp_{2g}(\bz / \ell^2 \ZZ)
		\]
		which, when taken modulo $\ell$, reduces to an inclusion 
		\begin{align}
			\label{iota-definition}
\iota_{\ell} \colon \Sp_{2g-2}(\bz / \ell \ZZ) \hookrightarrow \Sp_{2g}(\bz / \ell \ZZ). 
		\end{align}
		Thus, the group $\phi_{\ell,2g}^{-1}(\Sp_{2g-2}(\bz / \ell \ZZ))$ (see the diagram in~\eqref{roarrrr}) consists of matrices satisfying 
		\[
			\left[\begin{array}{c|c} 
				A_{(2g-2) \times (2g-2)} & B_{(2g-2) \times 2} \\ \hline
				C_{2\times(2g-2)} & D_{2 \times 2}
			\end{array}\right] \equiv 
			\left[\begin{array}{c|c}
				M & 0 \\ \hline 
				0 & \id_{2} 
			\end{array}\right] \pmod{\ell}
		\]
		where $M \in \Sp_{2g-2}(\bz / \ell \ZZ)$ and the blocks $A,B,C,D$ have the indicated sizes. For two such matrices, we have 
		\begin{align*} 
			\left[\begin{array}{c|c} 
			A_1 & B_1 \\ \hline
			C_1 & D_1
			\end{array}\right] \cdot
			\left[\begin{array}{c|c} 
			A_2 & B_2 \\ \hline
			C_2 & D_2
			\end{array}\right]
			&\equiv  
			\left[\begin{array}{c|c} 
			A_1A_2 + B_1C_2 & A_1B_2 + B_1D_2 \\ \hline
			C_1A_2 + D_1C_2 & C_1B_2 + D_1D_2
			\end{array}\right] \\
			&\equiv 
			\left[\begin{array}{c|c} 
			A_1A_2 & A_1B_2 + B_1D_2 \\ \hline
			C_1A_2 + D_1C_2 & D_1D_2
			\end{array}\right] \pmod{\ell^2},
		\end{align*} 
		where the last step follows because $B_1C_2 \equiv C_1B_2 \equiv 0 \pmod{\ell^2}$. Therefore, the map 
		\begin{align}
			\label{pi-definition}	
			\pi \colon \phi_{\ell,2g}^{-1}(\Sp_{2g-2}(\bz / \ell \ZZ)) \twoheadrightarrow \Sp_{2g-2}(\bz / \ell^2 \ZZ) \quad \text{sending} \quad \left[\begin{array}{c|c} A & B \\ \hline C & D \end{array}\right] \mapsto A
		\end{align}
		is a group homomorphism. This completes the definition of the maps $\iota_\ell$ and $\pi$, and it is apparent that the diagram in~\eqref{roarrrr} commutes. 
		
With this notation set, we now continue our proof of
Proposition~\ref{proposition:lifting-2}.
In Lemma~\ref{lemma:particular-matrices-in-preimage},
we show via explicit matrix multiplication that one of two particular matrices lies in
$H(\ell^2)$.
\begin{lemma}
	\label{lemma:particular-matrices-in-preimage}
	Suppose that
		\(
			\pi(H(\ell^2) \cap \phi_{\ell,2g}^{-1}(\Sp_{2g-2}(\bz / \ell \ZZ))) = \Sp_{2g-2}(\bz / \ell^2 \ZZ)
		\)
		and $H(\ell) = \Sp_{2g}(\mathbb Z/\ell \ZZ)$.
		Then, defining
		\begin{align} 
		\label{droptotheground}	\matone&\defeq \id_{2g} + \ell \cdot \left[\begin{array}{c|c} \begin{array}{cccc} 
			-1 & -1 & 0 & 0 \\ 
			0 & 1 & 0 & 0 \\
			0 & 0 & 0 & 0 \\
			0 & 0 & 0 & 0 
			\end{array} & 0_{4 \times (2g-4)} \\ \hline
			0_{(2g-4) \times 4} & 0_{(2g-4) \times (2g-4)}
			\end{array}\right] \quad \text{and} \\
		\label{pushofabutton}	\mattwo& \defeq \id_{2g} + \ell \cdot \left[\begin{array}{c|c} \begin{array}{cccc} 
			-1 & -1 & 1 & 0 \\ 
			0 & 1 & 0 & 0 \\
			0 & 0 & 0 & 0 \\
			0 & -1 & 0 & 0 
			\end{array} & 0_{4 \times (2g-4)} \\ \hline
			0_{(2g-4) \times 4} & 0_{(2g-4) \times (2g-4)}
			\end{array}\right],
		\end{align} 
we have that either $\matone$ or $\mattwo$ lies in $H(\ell^2)$.	
\end{lemma}
\begin{proof}
Since we are assuming
		\(
			\pi(H(\ell^2) \cap \phi_{\ell,2g}^{-1}(\Sp_{2g-2}(\bz / \ell \ZZ))) = \Sp_{2g-2}(\bz / \ell^2 \ZZ)
		\),
\mbox{it follows that}
		\[
			\id_{2g-2} + 
			\ell \cdot \left[\begin{array}{c|c} 
				\begin{array}{cc} 0 & 0 \\ 1 & 0 \end{array} & 0_{2 \times (2g-4)} \\\hline
				0_{(2g-4) \times 2} & 0_{(2g-4) \times (2g-4)}
			\end{array}\right]
		\]
		lies in $\pi(H(\ell^2) \cap \phi_{\ell,2g}^{-1}(\Sp_{2g-2}(\bz / \ell \ZZ)))$.
        It follows that $H(\ell^2)$ contains an element of the form $M = \id_{2g} + 
			\ell U$ for
		\[
			U = \left[\begin{array}{c|c|c} 
			\begin{array}{cc} 0 & 0 \\ 1 & 0 \end{array} & 0_{2 \times (2g-4)} & A_{2 \times 2} \\ \hline
			0_{(2g-4) \times 2} & 0_{(2g-4) \times (2g-4)} & B_{(2g-4) \times 2} \\ \hline
			A'_{2 \times 2} & B'_{2 \times (2g-4)} & C_{2 \times 2}
			\end{array}\right], 
		\]
		where there is a linear relation between $A_{2 \times 2}$ and $A'_{2 \times 2}$, as well as a linear relation between $B_{(2g-4) \times 2}$ and $B'_{2 \times (2g-4)}$, imposed by the symplectic constraint $M^T J_{2g} M = J_{2g}$. For any $M \in H(\ell^2)$, the group $H(\ell^2)$ also contains 
		\[
			M^{-1}(\id_{2g} + \ell U)M = \id_{2g} + \ell M^{-1} U M, 
		\]
		where the right-hand-side only depends on the reduction of $M$ modulo $\ell$. Since $H(\ell^2)$ surjects onto $\Sp_{2g}(\bz / \ell \ZZ)$, the matrix $M \pmod{\ell}$ ranges over all elements of $\Sp_{2g}(\bz / \ell \ZZ)$. With this in mind, take a matrix $M$ given by
		\begin{align*} 
			M &\defeq
			\left[\begin{array}{c|c|c}
				\begin{array}{cc} 1 & 1 \\ 0 & 1 \end{array} & 0_{2 \times (2g-4)} & 0_{2\times2} \\ \hline
				0_{(2g-4) \times 2} & \id_{2g-4} & 0_{(2g-4) \times 2} \\ \hline
				0_{2 \times 2} & 0_{2 \times (2g-4)} & \id_2 
			\end{array}\right] \pmod{\ell}
		\end{align*} 
		so that $H(\ell^2)$ contains 
	\begin{align}
			\label{conjugated-matrix-for-lifting}
			M^{-1}(\id_{2g} + \ell U) M = 
			\id_{2g} + 
			\ell \cdot 
			\left[\begin{array}{c|c|c} 
			\begin{array}{cc} -1 & -1 \\ 0 & 1 \end{array} & 0_{2 \times (2g-4)} & \left[\begin{array}{cc} 0 & -1 \\ 0 & 0 \end{array}\right] \cdot A_{2 \times 2}  \rule[-3.75ex]{0pt}{0pt} \\  \hline
			0_{(2g-4) \times 2} & 0_{(2g-4) \times (2g-4)} & 0_{(2g-4)
        \times 2} \\ \hline \rule{0pt}{5ex}  
			 A'_{2 \times 2} \cdot
      \left[\begin{array}{cc} 0 & 1 \\ 0 & 0 \end{array}\right] & 0_{2
        \times (2g-4)} & 0_{2 \times 2}
			\end{array}\right]. 
		\end{align}
		Multiplying~\eqref{conjugated-matrix-for-lifting} on the left by 
		\(
			(\id_{2g} + \ell U)^{-1} \equiv \id_{2g} - \ell U \pmod{\ell^2}
		\)
		shows that $H(\ell^2)$ contains 
\begin{align*}
				N &\defeq (\id_{2g} + \ell U)^{-1}M^{-1}(\id_{2g} + \ell U) M \\
			&= \id_{2g} + 
			\ell \cdot 
			\left[\begin{array}{c|c|c} 
			\begin{array}{cc} -1 & -1 \\ 0 & 1 \end{array} & 0_{2 \times (2g-4)} & \left[\begin{array}{cc} 0 & -1 \\ 0 & 0 \end{array}\right] \cdot A_{2 \times 2} \rule[-3.75ex]{0pt}{0pt} \\   \hline
			0_{(2g-4) \times 2} & 0_{(2g-4) \times (2g-4)} & 0_{(2g-4)
        \times 2} \\ \hline \rule{0pt}{5ex}  
			 A'_{2 \times 2} \cdot
      \left[\begin{array}{cc} 0 & 1 \\ 0 & 0 \end{array}\right] & 0_{2
        \times (2g-4)} & 0_{2 \times 2}
			\end{array}\right]. 
\end{align*}
Conjugating $N$ by any matrix of the form
		\[
		 P \defeq \left[\begin{array}{c|c|c} 
			\id_2 & 0_{2 \times (2g-4)} & 0_{2 \times 2} \\ \hline
			0_{(2g-4) \times 2} & \id_{2g-4} & 0_{(2g-4) \times 2} \\ \hline
			0_{2 \times 2} & 0_{2 \times (2g-4)} & M_{2 \times 2} 
			\end{array}\right] \pmod{\ell},
		\]
where $M_{2 \times 2} \in \Sp_2(\mathbb Z/\ell^2 \ZZ)$,
		results in the matrix 
		\begin{align*}
			N' &\defeq P^{-1}NP \\
			&=
			\id_{2g} + 
			\ell \cdot 
			\left[\begin{array}{c|c|c} 
			\begin{array}{cc} -1 & -1 \\ 0 & 1 \end{array} & 0_{2 \times (2g-4)} & \left[\begin{array}{cc} 0 & -1 \\ 0 & 0 \end{array}\right] \cdot A_{2 \times 2} \cdot M_{2 \times 2} \rule[-3.75ex]{0pt}{0pt} \\  \hline
			0_{(2g-4) \times 2} & 0_{(2g-4) \times (2g-4)} & 0_{(2g-4)
        \times 2} \\ \hline \rule{0pt}{5ex}  
			M_{2 \times 2}^{-1} \cdot A'_{2 \times 2} \cdot
      \left[\begin{array}{cc} 0 & 1 \\ 0 & 0 \end{array}\right] & 0_{2
        \times (2g-4)} & 0_{2 \times 2}
			\end{array}\right].
		\end{align*}
By choosing $M_{2 \times 2}$ judiciously, we may arrange that        
\[\left[\begin{array}{cc} 0 & -1 \\ 0 & 0 \end{array}\right] \cdot A_{2\times 2} \cdot M_{2 \times 2} \quad = \quad \left[\begin{array}{cc} 1 & 0 \\ 0 & 0 \end{array}\right] \quad \text{or} \quad \left[\begin{array}{cc} 0 & 0 \\ 0 & 0 \end{array}\right], \]
depending on whether the bottom row of $A_{2 \times 2}$ is nonzero or zero, respectively. Upon conjugating by the matrix 
		\[
     \left[\begin{array}{c|c} \id_2 & 0_{2 \times (2g-2)} \\ \hline  0_{(2g-2) \times 2} & \begin{array}{c|c|c} 0_{2 \times 2} & 0_{2 \times (2g-6)} & \id_2 \\ \hline
			0_{(2g-6) \times 2} & \id_{2g-6} & 0_{(2g-6) \times 2} \\ \hline
			\id_2 & 0_{2 \times (2g-6)} & 0_{2 \times 2} \end{array} \end{array}\right] \in \Sp_{2g}(\bz / \ell \ZZ), 
		\]
		we conclude that $H(\ell^2)$ contains either $\matone$
		or $\mattwo$.
	\end{proof}
Next, in Lemma~\ref{lemma:code}, we show that we can conjugate
the matrices $\matone$ and $\mattwo$ from Lemma~\ref{lemma:particular-matrices-in-preimage} to obtain all of $\Sp_4(\mathbb Z/\ell \ZZ)$.
	\begin{lemma}\label{lemma:code}
		Let $\ell = 2$ or $3$. Let $\ol{\matone}$ and $\ol{\mattwo}$ denote the upper-left $4 \times 4$ blocks of $\matone$ and $\mattwo$, respectively. Then the sets 
		\begin{align*} 
\left\{  M^{-1} 
	\ol{\matone} M : M \in \Sp_{4}(\bz / \ell \ZZ) \right\} \quad \text{and}\quad
\left\{ M^{-1} 
	\ol{\mattwo} M : M \in \Sp_{4}(\bz / \ell \ZZ) \right\} 
		\end{align*} 
		are both spanning sets for $1 + \ell \cdot \mf{sp}_{4}(\bz / \ell \bz)$. 
	\end{lemma} 
	\begin{proof}
    The following {\tt Magma} code verifies that each of the sets defined in the lemma statement span $\mf{sp}_{2g}(\bz / \ell \ZZ)$.
\vspace*{0.1in}
\begin{adjustwidth}{0.5in}{0in}
\begin{alltt}
for l in [2, 3] do
    Z := Integers();
    G := GL(4,quo<Z|l*l>);
    A := elt<G| 1,0,0,0, 1,1,0,0, 0,0,1,0, 0,0,0,1>;
    B := elt<G| 1,1,0,0, 0,1,1,0, 1,1,1,1, 1,1,0,1>;
    H := sub<G|A,B>;
    grp, f := ChangeRing(G, quo<Z|l>);
    Lie := Kernel(f) meet H;
    M := elt<G| 1-l,-l,0,0, 0,1+l,0,0, 0,0,1,0, 0,0,0,1>;
    N := elt<H| 1-l,-l,l,0, 0,1 + l,0,0, 0,0,1,0, 0,-l,0,1>;
    #sub<H|Conjugates(H,M)> eq #Lie;
    #sub<H|Conjugates(H,N)> eq #Lie;
end for;
	\end{alltt}
   \end{adjustwidth}
\vspace*{0.1in}
This concludes the proof of Lemma~\ref{lemma:code}
	\end{proof} 
    
We now have the tools to prove Proposition~\ref{proposition:lifting-2}.
\begin{proof}[Proof of Proposition~\ref{proposition:lifting-2}]
	The base case $g = 2$ is the content of Lemma~\ref{lemma:base-case-low-lifting}.
Now take $g \ge 3$, and suppose the result holds for $g - 1$.	
We shall now use the inductive hypothesis to show that $\ker \phi_{\ell,2g} \subset H(\ell^2)$, which would imply that $H(\ell^2) = \Sp_{2g}(\bz / \ell^2 \ZZ)$. 
		Since $H(\ell^2)$ surjects onto $\Sp_{2g}(\bz / \ell \ZZ)$, we have that the group $H(\ell^2) \cap \phi_{\ell,2g}^{-1}(\Sp_{2g-2}(\bz / \ell \ZZ))$ surjects onto $\Sp_{2g-2}(\bz / \ell \ZZ)$, and therefore so does the group
		\(
			\pi(H(\ell^2) \cap \phi_{\ell,2g}^{-1}(\Sp_{2g-2}(\bz / \ell \ZZ))). 
		\)
		By the inductive hypothesis, 
				\(
			\pi(H(\ell^2) \cap \phi_{\ell,2g}^{-1}(\Sp_{2g-2}(\bz / \ell \ZZ))) =\Sp_{2g-2}(\bz / \ell^2 \ZZ), 
		\)
		so upon applying Lemma~\ref{lemma:particular-matrices-in-preimage}, we deduce that either $\matone$ or $\mattwo$ lies in
$H(\ell^2)$.
Now, conjugating by matrices of the form 
		\[
			\left[\begin{array}{c|c}M_{4 \times 4} & 0_{4 \times (2g-4)} \\ \hline 0_{(2g-4) \times 4} & \id_{2g-4} \end{array}\right]\pmod{\ell}
		\]
where $M_{4 \times 4} \in \Sp_{4}(\bz / \ell \ZZ)$ serves to conjugate the upper-left $4 \times 4$ of $\matone$ and $\mattwo$ by $M_{4 \times 4}$. 
By Lemma~\ref{lemma:code}, we have that $\mf{sp}_{4}(\mathbb Z/\ell^2 \ZZ) \subset H(\ell^2)$, embedded as
		\begin{equation}\label{abovesubspace}
			\id_{2g} + \ell \cdot 
			\left[\begin{array}{c|c} 
			\mf{sp}_4(\bz / \ell \ZZ) & 0_{4 \times (2g-4)} \\ \hline
			0_{(2g-4) \times 4} & 0_{(2g-4) \times (2g-4)}
			\end{array}\right]. 
		\end{equation}
Construct $Q \in \Sp_{2g}(\bz / \ell \ZZ)$ by taking any $g \times g$ permutation matrix and replacing each $1$ with an $\id_2$-block. Conjugating the subspace in~\eqref{abovesubspace} by various such $Q$ shows that $H(\ell^2)$ contains $\ker \phi_{\ell,2g}$.
This can be seen by a straightforward argument involving
choosing a basis for $\mf{sp}_{2g}(\mathbb Z/\ell \ZZ)$
whose elements are elementary matrices or sums
of \mbox{two elementary matrices.}
\end{proof} 
	
\subsection{Finishing the Proof}
    We are now in position to complete the proof of Theorem~\ref{jamesdreamcometrue}

	\begin{proof}[Proof of Theorem~\ref{jamesdreamcometrue}]
		We split into three cases: 
		\begin{itemize} 
			\item Suppose $\ell \ge 5$. Then Lemma~\ref{proposition:lifting-l} implies $H(\ell^2) = \Sp_{2g}(\bz / \ell^2 \ZZ)$. 
			\item Suppose $\ell = 3$. Then Lemma~\ref{proposition:lifting-2} implies $H(9) = \Sp_{2g}(\bz / 9 \ZZ)$. 
			\item Suppose $\ell = 2$. Then Lemma~\ref{proposition:lifting-2} implies $H(4) = \Sp_{2g}(\bz / 4 \ZZ)$, and Lemma~\ref{proposition:lifting-4} implies $H(8) = \Sp_{2g}(\bz / 8 \ZZ)$. 
		\end{itemize} 
	Combining the above results with Lemma~\ref{proposition:lifting-l-squared} gives the desired conclusion.
	\end{proof}

\section*{Acknowledgments} 	

\noindent This research was supervised by Ken Ono and David Zureick-Brown at the Emory University Mathematics REU and was supported by the National Science Foundation (grant number DMS-1557960). We would like to thank David Zureick-Brown for suggesting the problem that led to the present article and for offering us his invaluable advice and guidance. We would also like to thank David Zureick-Brown for providing the intuition behind the proof of Lemma~\ref{proposition:lifting-l}. We would like to thank Jackson Morrow, Ken Ono, and David Zureick-Brown for making several helpful comments regarding the composition of this article. We would like to acknowledge Michael Aschbacher and Nick Gill for their helpful advice. We used {\tt Magma} and \emph{Mathematica} for explicit calculations.

\bibliographystyle{alpha}
\bibliography{bibfile}

\end{document}